\documentclass[11pt,reqno,a4paper]{amsart}
\usepackage[latin1]{inputenc}
\usepackage[latin1]{inputenc}
\usepackage{a4wide}
\usepackage{fancyhdr}
\usepackage{amsfonts,amssymb,amsmath,amsthm,latexsym,indentfirst}
\usepackage{epsfig}
\usepackage{mathabx}
\usepackage{enumerate}
\usepackage{amsfonts}
\usepackage{latexsym}
\usepackage{amssymb}
\usepackage{amsmath}
\usepackage{textcomp}
\usepackage{color,graphicx}
\usepackage{graphicx}

\newtheorem{theorem}{Theorem}[section]
\newtheorem{proposition}[theorem]{Proposition}
\newtheorem{remark}[theorem]{Remark}
\newtheorem{lemma}[theorem]{Lemma}

\newcommand{\N}[1][]{\ensuremath{{\mathbb{N}^{#1}} }}
\newcommand{\Z}[1][]{\ensuremath{{\mathbb{Z}^{#1}} }}

\newcommand{\R}[1][]{\ensuremath{{\mathbb{R}^{#1}} }}

\newcommand{\re}{\mathbb{R}}

\newcommand{\E}{\mathcal{E}}

\title[A degenerated Zakharov system]{A remark on the well-posedness of the degenerated Zakharov system}

\author[V.  Barros and F. Linares]{Vanessa Barros and Felipe Linares}
\address{Universidade Federal da Bahia, Instituto de matem\'atica, Av. Adhemar de Barros , Ondina, 40170-110, Salvador,
Bahia, Brazil}
\email{vbarros@impa.br}
\address{IMPA, Estrada Dona Castorina 110,  Rio de Janeiro 22460-320, Brazil}
\email{linares@impa.br}
\dedicatory{Dedicated to Gustavo Ponce for his 60$^{\,th}$}

\keywords{Zakharov systems, local well-posedness}

\subjclass[2000]{35D05, 35E15, 35Q35}

\begin{document}

\maketitle

\begin{abstract} We extend the local well-posedness theory for the Cauchy problem associated to a degenerated Zakharov system.
The new main ingredients are the derivation of  Strichartz and maximal function norm estimates  for the linear solution of
a Schr\"odinger type equation with missing dispersion in one direction. The result here improves the one in  \cite{LiPoS}.
\end{abstract}

\section{Introduction}
We consider the initial value problem  associated to the degenerate Zakharov system

\begin{equation}\label{eqzd}
\begin{cases}
i (\partial_tE + \partial_zE )+ \Delta_\perp E  = nE,\quad  (x, y, z) \in \re^3,\ t > 0,\\
 \partial^2_{t}n-\Delta_\perp n  =  \Delta_\perp (|E|^2),\\
E(\cdot, 0)= E_0(\cdot),\;\; n(\cdot, 0)=n_0(\cdot), \;\; \partial_t n(\cdot, 0)=n_1(\cdot),
\end{cases}
\end{equation}
where  $\Delta_\perp=\partial^2_x + \partial^2_y$, $E$ is a complex-valued function, and $ n$ is a real-valued function.
The system \eqref{eqzd}  describes the laser propagation when the paraxial approximation is used and the effect of the
group velocity is negligible (\cite{R}). 

We use the term degenerate in the sense that there is no dispersion in the $z$-direction for the system in \eqref{eqzd}  
in contrast to
the well known  Zakharov system
\begin{equation}\label{eqz}
\begin{cases}
i \partial_t E + \Delta E =  nE,\quad  (x, y, z) \in \re^3,\ t > 0,\\
 \partial_t^2 n-\Delta n   = \Delta(|E|^2),\\
\end{cases}
\end{equation}
which was introduced in \cite{Z} to describe the long wave Langmuir turbulence in a plasma.

Regarding the IVP \eqref{eqzd}, Colin and Colin in \cite{CC} posed the question of the well-posedness. 
A positive answer was given by  Linares, Ponce and Saut in \cite{LiPoS},  showing the local well-posedness of
the IVP  \eqref{eqzd} in a suitable Sobolev space. The results proved in \cite{LiPoS} extended previous
ones for the Zakharov  system \eqref{eqz}, where transversal dispersion is taken into account (see \cite{OT}, \cite{GTV} and references
therein). However,  the system \eqref{eqzd} is quite different from the classical  Zakharov system \eqref{eqz} since the Cauchy problem
 for the periodic data exhibits strong instabilities of the Hadamard type implying ill-posedness (see \cite{CM}).

Our goal here is to extend the  local well-posedness  for the IVP \eqref{eqzd} to a larger functional space than that in \cite{LiPoS}.

Before describing our main result and the new ingredients used in its proof we proceed as in \cite{LiPoS} to study this problem.

First the IVP \eqref{eqzd} is reduced into the IVP associated to a single equation, that is,
\begin{equation}\label{eqzd1}
\begin{cases}
i (\partial_tE + \partial_zE )+ \Delta_\perp E  =  nE, \quad (x, y, z) \in \re^3 ,\,t > 0,\\
E(x, y, z, 0)= E_0(x, y, z),
\end{cases}
\end{equation}
where
\begin{equation*}
n(t)=N'(t)n_0+N(t)n_1+\int_0^t N(t-t') \Delta_\perp (|E(t')|^2)dt',
\end{equation*}
with
\begin{equation}\label{defN}
 N(t)f=(-\Delta_\perp)^{-1/2}\sin((-\Delta_\perp)^{1/2}t)f,
\end{equation}
and
\begin{equation}\label{defN'}
 N'(t)f=\cos((-\Delta_\perp)^{1/2}t)f,
\end{equation}
where  $(-\Delta_\perp)^{1/2}f=((\xi_1^2+\xi_2^2)^{1/2}\widehat{f})^{\vee}$.\\

Then it is  considered  the integral equivalent formulation of the IVP \eqref{eqzd1}, that is, 
\begin{equation}\label{eqint1}
\begin{split}
 E(t)=&\E(t)E_0 +\int_{0}^{t}\E(t-t')(N'(t')n_0+N(t')n_1)E(t')dt'\\
&+\int_{0}^{t}\E(t-t')\big(\int_{0}^{t'}N(t'-s)\Delta_\perp(|E(s)|^2)\,ds\big)E(t')\,dt',
\end{split}
\end{equation}
where $ \E(t)$ denotes the unitary group associated to the linear problem to \eqref{eqzd1} given by
\begin{equation}\label{slei}
 \E(t)E_0=\big(e^{-it(\xi_1^2+\xi_2^2+\xi_3)}\widehat{E_0}(\xi_1, \xi_2, \xi_3)\big)^{\vee}.
\end{equation}

A smoothing effect for the unitary group $\E(t)$ similar to the one obtained for 
solutions of the linear  Schr\"odinger equation  was proved in \cite{LiPoS} (see  Proposition \ref{p1} below). This was the main tool used
there to establish local well-posedness  via contraction principle  in the following functional space
\begin{equation}\label{defi1}
 \widetilde{H}^{2j+1}(\re^3)=\{f \in H^{2j+1}(\re^3),D_x^{1/2}{\partial}^{\alpha} f, \ D_y^{1/2}{\partial}^{\alpha}
 f \in L^2(\re^3),\ |\alpha|\leq 2j+1,\ j\in \N \},
 \end{equation}
where $\alpha\in (\mathbb Z^{+})^3$ is a multiindex, $D^{1/2}_xf=(|\xi_1|^{1/2}\hat{f})^\vee$ and $D_y^{1/2}f=(|\xi_2|^{1/2}\hat{f})^\vee$.

Roughly the result in \cite{LiPoS} guarantees the local well-posedness in  $\widetilde{H}^{2j+1}(\re^3)$, $j\ge 2$, for data $E_0\in \widetilde{H}^{2j+1}(\re^3)$,
$n_0\in H^{2j}(\R^3)$ and $n_1\in H^{2j-1}(\R^3)$ with $\partial_z n_1\in H^{2j-1}(\R^3)$, where $H^s(\R^3)$ is the usual Sobolev space.

To improve the previous result obtained in \cite{LiPoS} we derive  two new estimates for solutions of the linear problem. The first one is the following  Strichartz estimate,
\begin{align}
 \|\E(t)f\|_{L^q_tL^p_{xy}L^2_z}\leq c\|f\|_{L^2_{xyz}},
\end{align}
where $2/q=1-2/p$, $2\le p<\infty$.

We can observe that the lack of dispersion in the $z$-direction is reflected in the estimate above. The proof uses the explicit Fourier transform of $e^{itx^2}$
and the usual method to prove Strichartz estimates for the linear Schr\"odinger equation.

The second new estimate for solutions of the linear problem is the following maximal function estimate
\begin{equation}\label{desig2}
 \|\E(t)f\|_{L^2_xL^{\infty}_{yzT}}\leq c(T, s)\|f\|_{H^s(\re^3)},\;\; s>3/2.
\end{equation}
The argument to prove \eqref{desig2} follows the ideas in \cite{KZ}, where they obtained a $L^4_x$-maximal function estimates for solutions of the linear
 problem associated to the modified Kadomtsev-Petviashvili (KPI) equation.
 
 \begin{remark} It is not clear whether the estimate \eqref{desig2} is sharp. In Proposition \ref{propositionB} below we show that this estimate is false
 in $H^s(\R^3)$  for $s<1$.
 \end{remark}

To state our result we shall slightly modify the space $\widetilde{H}^{2j+1}(\re^3)$ defined in \eqref{defi1}. We define
\begin{equation}\label{newH}
 \widetilde{H}^2(\re^3)=\{f \in H^2(\re^3),D_x^{1/2}{\partial}^{\alpha} f, \ D_y^{1/2}{\partial}^{\alpha}
 f \in L^2(\re^3),\ |\alpha|= 2\}.
 \end{equation}

With this notation, the main result here reads as:

 \begin{theorem}\label{T2}
 For initial data $(E_0, n_0, n_1)$ in $\widetilde{H}^{2}(\re^3)\times H^{2}(\re^3)\times
 H^{1}(\re^3) \text{ and } \partial_zn_1
 \in
 H^{1}(\re^3),$ there exist
 $T>0$ and a unique solution $E$ of the integral equation \eqref{eqint1} such that 
\begin{equation}
E \in C([0, T]:\widetilde{H}^{2}(\re^3)),\label{ww1}
\end{equation}
\begin{equation}\label{ww2}
 \sum_{|\alpha|=2}\big( \|\partial_x{\partial}^{\alpha}E \|_{L^{\infty}_xL^2_{yzT}}+ 
 \|\partial_y{\partial}^{\alpha}E\|_{L^{\infty}_yL^2_{xzT}}\big)<\infty,
\end{equation}
\begin{equation}\label{ww3}
\|E \|_{L^{2}_xL^{\infty}_{yzT}} + \|E \|_{L^{2}_yL^{\infty}_{xzT}}<\infty,
\end{equation}
and 
\begin{equation}\label{ww4}
X_T(E)<\infty
\end{equation}
where $X_T(\cdot)$ is defined in \eqref{strichartznorms} below.

Moreover, there exists a neighborhood $V$ of $(E_0, n_0, n_1)\in 
\widetilde{H}^{2}(\re^3)\times H^{2}(\re^3) \times H^{1}(\re^3)$ 
such that the map  $\mathcal{F}:(E_0, n_0, n_1)\mapsto E(t)$ from $V$
into the class defined by \eqref{ww1}-\eqref{ww4} is smooth.

One also has that
\begin{equation*}
 n \in C([0, T]:H^{2}(\re^3)).
\end{equation*}
\end{theorem}

\begin{remark}   Observe that in comparison with the result in \cite{LiPoS}  we could considerably weaken the regularity required
to prove local well-posedness for the IVP \eqref{eqzd}.  Notice also that it would be possible to lower  the regularity a little further because 
the maximal function works well in $H^s(\R^3)$, $s>3/2$. 
\end{remark}

\begin{remark} The Strichartz estimates were essential in our analysis. It may be possible to use them in the Bourgain spaces
framework to obtain better results (see for instance \cite{bhht}, \cite{GTV},  and references therein for the Zakharov system). Regarding global well-posedness, we do 
not know any conserved quantity that might be useful to extend globally
the local results.
\end{remark}

The plan of the paper is the following. In Section \ref{linearestimates} we prove the new linear estimates commented above
and recall some known ones established in \cite{LiPoS}. Some useful lemmas will also be presented in this section.  
In Section \ref{nonlinearestimates}  we establish  estimates involving the nonlinear term that allow us to 
simplify the exposition of the 
proof of the main result. Finally our main result will be proved in Section \ref{proofoftheoremT2}.

Before leaving this section we introduce the notation used throughout the paper.  We use standard 
notation in Partial Differential
Equations. In addition we will use $c$  to denote various constants that may change from line to line.

Let $x=(x_1,x_2,x_3)$ and $\xi=(\xi_1,\xi_2,\xi_3)$. For $f=f(x,t)\in \mathcal S(\R^4)$,  $\widehat{f}$ will denote 
its Fourier transform in space, whereas ${\widehat{f}}^{(x_ix_l)}$, respectively $\widehat{f}^{(x_i)}$, will denote its
Fourier transform in the $x_ix_l$ and $x_i$ variables,
$i,l=1,2,3$. For $s\in \R$, we define the Bessel and Riesz potentials of order $-s$, $J^s_x$ and $D^s_x$, by
\begin{equation}\label{bessel}
\widehat{J^s_xf}=
(1+|\xi|^2)^{s/2}\widehat{f}    \text{\hskip10pt and \hskip10pt}     \widehat{D^s_xf}=|\xi|^s\widehat{f}.
\end{equation}
We also use the notation $J^s_{x_ix_l}$ and $J^s_{x_i}$ to denote the
operators 
\begin{equation}\label{besselpart}
\widehat{J^s_{x_ix_l}f}=(1+|(\xi_i,\xi_l)|^2)^{s/2}\widehat{f} \text{\hskip10pt and \hskip10pt}
\widehat{J^s_{x_i}f}=(1+|\xi_i|^2)^{s/2}\widehat{f}, \quad i,l=1,2,3.
\end{equation}

We introduce the next notation to set together all the terms involving the Strichartz norms in our analysis.
\begin{equation}\label{strichartznorms}
\begin{split}
X_T(f):&= \underset{|\alpha|=1}{\sum} \big(\|J^{1/4+}_zD^{1/2}_x\partial^{\alpha}f\|_{L^4_{xyT}L^2_{z}}
+\|J_z^{3/8+}\partial^{\alpha}f\|_{L^{8/3}_TL^8_{xy}L^2_{z}}+\|J_z^{1/2+}\partial^{\alpha}f \|_{L^{4}_{xyT}L^2_{z}}\big)\\
&+ \sum_{|\alpha|\leq1}\big(\|\partial_{x}\partial^\alpha f \|_{L^{4}_{xyT}L^2_{z}} + 
\|\partial_{y}\partial^\alpha f \|_{L^{4}_{xyT}L^2_{z}}\big).
\end{split}
\end{equation}

\section{Linear estimates}\label{linearestimates}

Consider the linear problem:
\begin{equation}\label{eqlinear}
\left\{\begin{array}{l}
\partial_t E + \partial_z E -i \Delta_\perp E  = 0,\quad \forall\, (x, y, z) \in \re^3,\ t > 0,\\
E(x, y, z, 0)= E_0(x, y, z).
\end{array}
\right.
\end{equation}
 
where $ \Delta_\perp=\partial^2_x + \partial^2_y$.

The solution of the linear IVP \eqref{eqlinear} is given by the unitary group $\E(t):H^s\rightarrow H^s$ such that
\begin{equation}\label{sol}
 E(t)= \E(t)E_0=\left(e^{-it(\xi_1^2+\xi_2^2+\xi_3)}\widehat{E_0}(\xi_1, \xi_2, \xi_3)\right)^{\vee}.
\end{equation}

\begin{proposition}\label{p1}
 The solution of the linear problem \eqref{eqlinear} satisfies
\begin{align}
\|D_x^{1/2}\E(t)f\|_{L^{\infty}_xL^2_{yzT}}\leq c \|f\|_{L^2_{xyz}}, \label{LE1}
\end{align}
\begin{align}
\|D_x^{1/2}\int_0^t\E(t-t')G(t')dt'\|_{L^{\infty}_TL^2_{xyz}}\leq c \|G\|_{L^1_xL^2_{yzT}},\label{LE2} 
\end{align}
and 
\begin{align}
\|\partial_x\int_0^t\E(t-t')G(t')dt'\|_{L^{\infty}_xL^2_{yzT}}\leq c \|G\|_{L^1_xL^2_{yzT}}\label{LE3}.
\end{align}
These estimates hold exchanging x and y. Here $D_x^{1/2}f=(2\pi|\xi_1|^{1/2}\hat{f})^{\vee}$.
\end{proposition}
\begin{proof}
 We refer to \cite{LiPoS} for a proof of this proposition.
\end{proof}

Now we give the precise statement of the inequality \eqref{desig2} and its proof. 

\begin{proposition}\label{propositionA}
 For $s>3/2$, and $T>0$ we have 
\begin{equation}\label{E2}
\|\E(t)E_0\|_{L_x^2L^\infty_{yzT}}\leq c(T,s) \|E_0\|_{H^s(\R^3)}.
\end{equation}
The same estimate holds exchanging x and y.
\end{proposition}

The proof of Proposition \ref{propositionA} is a direct consequence of the next lemma, as we shall see later.

\begin{lemma}\label{lemmaA}
For every $ T>0$ and $k \geq 0$, there exist a constant $c(T) > 0 $ and a positive function $H_{k,T}(\cdot)$ such that 
\begin{equation}\label{H}
\int_{0}^{+\infty} H_{k, T}(y) d\alpha \leq c(T)2^{3k}, 
\end{equation}
and
\begin{equation}\label{J}
|\int_{\R^3}e^{i(-t(\xi_1^2+\xi_2^2+\xi_3)+x\cdot\xi)}\overset{3}{\underset{i=1}{\prod}}\psi_i(\xi_i)\, d\xi\big| 
\leq H_{k, T}
(|x_1|), 
\end{equation}
for $|t| \leq T$ and $x=(x_1,x_2,x_3)$ and $\xi=(\xi_1,\xi_2,\xi_3)$ in $\re^3 $ where  $\psi_i(\xi_i)=\psi(2^{k+1}-|\xi_i|)$,
and $\psi$ denotes a  $C^{\infty}(\re)$ function such that $\psi=1$ for $x\geq 1$  and  $\psi=0$ for $x\leq0$.
\end{lemma}

To prove this lemma we will employ the argument introduced in \cite{Fa}

\begin{proof}
Denote by $J(x,y,z,t)$ the integral on the left-hand side in \eqref{J}. We can rewrite $J(t, x, y, z)$ as:
\begin{equation*}
J( x, y, z,t)=\overset{3}{\underset{i=1}{\prod}} J_i(x_i, t),
\end{equation*}
where
 \begin{equation*}
 J_i(x_i,t)=\int e^{i\varphi_i(\xi_i)} \psi_i(\xi_i) d\xi_i \text{\hskip10pt and \hskip10pt}
 \varphi_i(\xi_i)=(-t\xi_i^2+x_i\xi_i), \;\;i=1,2,
 \end{equation*}
 and 
 $$J_3=\int e^{i(-t\xi_3+z\xi_3)} \psi_3(\xi_3) d\xi_3,$$
 we have $|J|\leq|J_1\|J_2\|J_3|.$ 

Next we consider the following three cases:
\begin{itemize}
\item For $|x_1|<1$ we use the support of $\psi_j,\ j=1,2,3,$ and get $|J|\leq c\,2^{3k}$.\\

\item  For $|x_1| \geq \max \{ 1, 2^3 2^k t\}$. In this case $|x_1|\geq 4|\xi_1|t$ for $\xi_1$ in the support of $\psi_1,$ and so $|\varphi_1^{'}(\xi_1)|\geq |x_1|/2$.
Using integration by parts twice we get:
\begin{equation*}
J_1=\int e^{i\varphi_1}\Big(\dfrac{1}{\varphi_1'}\Big(\dfrac{\psi_1}{\varphi_1'}\Big)'\Big)'\, d\xi_1.
\end{equation*}
Now by the support of  $\psi_1$ and the inequalities  $|\varphi_1^{'}(\xi_1)|\geq |x_1|/2 \text{ and }\  |x_1|^{-1}\leq1$ we have:
\begin{equation*}
|J_1|\leq c(T) \int_{\{|\xi_1|\leq 2^{k+1}\}} \dfrac{1}{|x_1|^2} \,d\xi_1 \leq c(T) 2^{k} |x_1|^{-2} .
\end{equation*}
Then $|J|\leq 2^ {3k}c(T) |x_1|^{-2}$,  using the supports of $\psi_2$ and $\psi_3$.\\

\item For $1\leq|x_1|\leq 2^3 2^k |t|$.  Observe that in this case $t\geq 2^{-k-3}>0$ and $t^{-2}\leq c |x_1|^{-2} 2^{2k}$.
Since $|\varphi_1^{''}(\xi_1)|=2t>0$,  Van der Corput lemma (see \cite{LP2} for instance) implies $|J_1|\leq ct^{-1/2}$. Similarly, we have $|J_2|\leq c t^{-1/2}$.
Thus  $|J| \leq c t^{-1}2^k \leq c T t^{-2} 2^k \leq c 2^{3k}|x_1|^{-2}$ by using the support of  $\psi_3$.
\end{itemize}

Finally we define
\begin{equation*}
 H_{k, T}(\rho)=\begin{cases}
 c2^{3k} \text{\hskip42pt for\hskip10pt}  0\leq \rho< 1,\\
c(T)2^{3k} \rho^{-2} \text{\hskip10pt for\hskip10pt} 1\leq \rho,
\end{cases}
\end{equation*}
and this function satisfies \eqref{H} and \eqref{J}.
\end{proof}

\begin{remark}
Observe that Lemma \ref{lemmaA} still works if we change $\psi_j$ 
by $\psi_j\psi(|\xi_j|-2^{k}+1)$, $j=1,2$ or $3$.
\end{remark}

\begin{proof}[Proof of Proposition \ref{propositionA}]
Using the same notation as in Lemma \ref{lemmaA}, i.e., $\psi_j=\psi(2^{k+1}-|\xi_j|),\ j=1,2,3$,
we define the sequence $\{\widetilde{\psi}_k\}$ as follows:
\begin{equation*}
\widetilde{\psi}_0(\xi_1, \xi_2, \xi_3)=\psi(2-|\xi_1|)\psi(2-|\xi_2|)\psi(2-|\xi_3|),
\end{equation*}
and for $k\geq 1$,
\begin{equation*}
\bar{\psi}_k(\xi_1, \xi_2, \xi_3) =\overset{3}{\underset{i=1}{\sum}}\psi_1\psi_2\psi_3\psi(|\xi_i|-2^{k}+1).
\end{equation*}
Notice  that $\sum_{k\geq 0}\widetilde{\psi}_k=1$.

Now we define the operator $\widehat{B_kf}(\xi)=\bar{\psi}_k^{1/2}(\xi) \hat{f}(\xi)$, $\xi \in \R^3$.

Then it is not difficult to verify that
\begin{align}
 \|B_kf\|_{L^2}&\leq c2^{-ks}\|f\|_{H^s},\label{B} \\
\widehat{B_k^2f}&=\bar{\psi}_k\hat{f},\label{B1}\\
\sum_{k\geq 0}\E(t)B^2_kE_0&=\E(t)E_0 ,\label{B2}
\end{align}
and
\begin{equation} \label{L}
|\int_{-T}^{T}(\E(t-\tau)(B_k^2 g(\cdot,\tau))(x, y, z) d\tau| \leq c(H_{k, T}(|\cdot|)\ast\int_{-T}^{T}\int \int|g(\tau, \cdot, y, z)| d\tau dy dz)(x),
\end{equation}
for $|t| \leq T$ and $g \in C_{0}^{\infty}(\re^4)$.

Since from this point on the argument to complete the proof of the proposition is well understood (see for instance
\cite{KPV1}) we will omit it. Thus the result follows.
\end{proof}

Now, following ideas from Kenig and Ziesler for the KPI equation (see\cite{KZ}), we show that \eqref{E2} does not hold for 
$s<1$.

\begin{proposition}\label{propositionB}
 For each $s< 1 $ there exists $F_0$ such that
$$ \|\E(t)F_0\|_{L_x^2L^\infty_{yzT}} \geq c(T,s) \|F_0\|_{H^s}.$$
\end{proposition}

\begin{proof}
 Suppose that \eqref{E2} is true and define $\hat{E_0}(\xi)=\hat{\theta}(\frac{\xi}{2^k})$, 
 where $k\in\N$ and $\hat{\theta}
 \in C^{\infty}_0$ is such that 
\[\hat{\theta}(\xi)=\left\{ \begin{array}{ccl} 1 & \text{on}& \{\xi \in \re^3; 1\leq|\xi|\leq2\},\\0 &\text{ on }& \{\xi \in \re^3;
|\xi|\leq1/2\}\cup \{\xi \in \re^3; |\xi|\geq4\}.\end{array}\right.\]
So by change of variables
\begin{align*}
 \|E_0\|_{H^s}=&(\int_{\{\frac{1}{2}\leq\frac{|\xi|}{2^k}\leq4\}}(1+|\xi|^2)^s
|\hat{\theta}(\frac{\xi}{2^k})|^2d\xi)^{\frac{1}{2}}=(\int_{\{\frac{1}{2}\leq|\xi|\leq4\}}(1+2^{2k}|\xi|^2)^s |\hat{\theta}(\xi)|^2 2^{3k}d\xi)^{\frac{1}{2}}\\
\leq& 2^{3k/2+ks}(\int_{\{\frac{1}{2}\leq|\xi|\leq4\}}(2^{-2k}+|\xi|^2)^s |\hat{\theta}(\xi)|^2 d\xi)^{\frac{1}{2}}
\leq  2^{3k/2+ks}c(s)(\int_{\{\frac{1}{2}\leq|\xi|\leq4\}} |\hat{\theta}(\xi)|^2 d\xi)^{\frac{1}{2}}\\
\leq & 2^{3k/2+ks}c(s).
\end{align*}
Next, we estimate $\E(t)E_0$. Again by changing variables we have
\begin{equation*}
\begin{split}
 (\E(t)E_0)(x, y, z)=&\int_{\{\frac{1}{2}\leq\frac{|\xi|}{2^k}\leq4\}} e^{i(x\xi_1+y\xi_2+z\xi_3+t(\xi_1^2+\xi_2^2+\xi_3))}
\hat{\theta}(\frac{\xi}{2^k}) d\xi\\
= &2^{3k}\int_{\{\frac{1}{2}\leq|\xi|\leq4\}} e^{ix\bar{\xi}}e^{is}
\hat{\theta}(\xi) d\xi,
\end{split}
\end{equation*}
where $\bar{\xi}=2^k\xi_1,\ s=y2^k\xi_2+z2^k\xi_3+t(2^{2k}\xi_1^2+2^{2k}\xi_2^2+2^k\xi_3).$

Now, by Taylor's expansion
\begin{equation*}
\begin{split}
|(\E(t)E_0)(x, y, z)|&=2^{3k}|\int_{\{\frac{1}{2}\leq|\xi|\leq4\}} e^{ix\bar{\xi}} e^{is}\hat{\theta}(\xi) d\xi|\\
 &\geq  2^{3k}|\int_{\{\frac{1}{2}\leq|\xi|\leq4\}}(\cos(x\bar{\xi})\cos(s)-\sin(x\bar{\xi})\sin(s))\hat{\theta}(\xi) d\xi|\\
&\geq  2^{3k}|\int_{\{\frac{1}{2}\leq|\xi|\leq4\}}[(1-\frac{(x\bar{\xi})^2}{2}+r(x\bar{\xi}))(1-\frac{s^2}{2}+r(s))+\\
&-((x\bar{\xi})^2-r_1(x\bar{\xi}))(s-r_1(s))]\hat{\theta}(\xi) d\xi|\\
&\geq 2^{3k}|\int_{\{\frac{1}{2}\leq|\xi|\leq4\}}[1-\eta(x, s, \bar{\xi})+\rho(x, s, \bar{\xi})]\hat{\theta}(\xi) d\xi|,
\end{split}
\end{equation*}
where 
$$\eta(x, s, \bar{\xi})=\frac{(x\bar{\xi})^2}{2}+\frac{s^2}{2}+\frac{s^2r(x\bar{\xi})}{2}+
\frac{{(x\bar{\xi})}^2r(s)}{2}+sx\bar{\xi}+r_1(x\bar{\xi})r_1(s),$$
$$\rho(x, s, \bar{\xi})=r(x\bar{\xi})+\frac{s^2(x\bar{\xi})^2}{2}+r(s)+r(x\bar{\xi})r(s)
+x\bar{\xi}r_1(s)+sr_1(x\bar{\xi}),$$
\begin{equation*}
r(\cdot)=(\cdot)^4-(\cdot)^6+(\cdot)^8-\ldots\text{\hskip10pt and \hskip10pt }
r_1(\cdot)=(\cdot)^3-(\cdot)^5+(\cdot)^7-\ldots.
\end{equation*}

If we choose $ 0<\delta\ll 1$ and take $|x|\leq \delta 2^{-k}, \ y,\ z\simeq \delta 2^{-k},\ t\simeq \delta 2^{-2k}$,\\ 
then $s, \ x\bar{\xi}\simeq O(\delta),\ 0<r(s),\ r_1(s),\ r(x\bar{\xi}),\ r_1(x\bar{\xi})\ll 1$ and $1-\eta(x, s, \bar{\xi})>c>0.$\\
So, 
 $$\|(\E(t)E_0)(x, y, z)|\geq c 2^{3k}|\int_{\{\frac{1}{2}\leq|\xi|\leq4\}}A \cdot \hat{\theta}(\xi) d\xi|\geq
c  2^{3k}|\int_{\{1\leq|\xi|\leq2\}} A \cdot 1 \ d\xi|\geq c 2^{3k}.$$
Then, 
$$ \|(\E(t)E_0)\|_{L_x^2L^\infty_{yzT}}\geq(\int_{|x|\leq \delta 2^{-k}}( \underset{\underset{y,z\simeq \delta 2 ^{-k}}{t\simeq \delta2^{-2k}}}{\sup}|E(t)E_0|)^2 dx)^{1/2}\geq 2^{3k}2^{-k/2}=2^{5k/2}.$$
Finally, we have
$$  c \ 2^{5k/2}\leq\|(\E(t)E_0)\|_{L_x^2L^\infty_{yzT}} \leq \|E_0\|_{H^s}\leq 2^{3k/2+ks} \ \forall k \in \N,$$
 which implies $s\geq 1$.
\end{proof}

Now we establish Strichartz estimates to the linear problem \eqref{eqlinear}. Before that we state and prove an essential lemma:

\begin{lemma}\label{lj}
 If $t\neq 0,\frac{1}{p}+\frac{1}{p'}=1$ and $p'\in [1,2],$ then the group $\E(t)$ defined in \eqref{sol} is a continuous linear
 operator from $L^{p'}_{xy}L^2_z(\re^3)$ to $L^p_{xy}L^2_z(\re^3)$ and
$$\|\E(t)f\|_{L^p_{xy}L^2_z}\leq \frac{c}{|t|^{(\frac{1}{p'}-\frac{1}{p})}}\|f\|_{L^{p'}_{xy}L^2_z}.$$
\end{lemma}
\begin{proof}
 From Plancherel's theorem we have that
\begin{align}\label{q1}
 \|\E(t)f\|_{L^2_{xy}L^2_z}=\|\E(t)f\|_{L^2}=\|e^{-it(\xi_1^2+\xi_2^2+\xi_3)}\hat{f}\|_{L^2}=
\|\hat{f}\|_{L^2}=\|f\|_{L^2_{xy}L^2_z}.
\end{align}
Using Fourier's transform properties we obtain
\begin{align*}
 (\E(t)f)(x,y,z)&=(e^{-it(\xi_1^2+\xi_2^2+\xi_3)}\widehat{f}
 (\xi_1,\xi_2,\xi_3))^{\vee}(x,y,z)\\
&=\big(e^{-it\xi_3}(e^{-it(\xi_1^2+\xi_2^2)}\widehat{f}(\xi_1,\xi_2,\xi_3))^{{\vee}_{(x_1x_2)}}(x,y,\cdot)\big)^{{\vee}_{(x_3)}}
(\cdot,\cdot,z)\\
&=\big(e^{-it\xi_3}((e^{-it(\xi_1^2+\xi_2^2)})^{\vee{(x_1x_2)}}\ast_{x_1x_2}\widehat{f}^{(x_3)}(\xi_1,\xi_2,\xi_3))(x,y,\cdot)
\big)^{\vee_{(x_3)}}(\cdot,\cdot,z)\\
&=\big(e^{-it\xi_3}(\frac{e^{i(\xi_1^2+\xi_2^2)/4|t|}}{4 \pi t}\ast_{x_1x_2}\widehat{f}^{(x_3)}(\xi_1,\xi_2,\xi_3))(x,y,\cdot)
\big)^{\vee_{(x_3)}}(\cdot,\cdot,z)\\
&=\big(e^{-it\xi_3}g(x,y,\xi_3)\big)^{\vee_{(x_3)}}(\cdot,\cdot,z),
\end{align*}
where $$g(x,y,\cdot)=\big(\frac{e^{i(\xi_1^2+\xi_2^2)/4|t|}}{4 \pi t}
\ast_{x_1x_2}\widehat{f}^{(x_3)}(\xi_1,\xi_2,\xi_3)\big)(x,y,\cdot),$$ and $\ast_{x_1x_2}$ is the convolution in the first two variables, i.e.,
$$(f_1\ast_{x_1x_2}f_2)(x,y,z)=\int_{\re^2}f_1(x-x_1,y-x_2,z)f_2(x_1,x_2,z)dx_1dx_2.$$
By Plancherel's theorem and Minkowski's inequality we have
\begin{equation*}
\begin{split}
 \|\E(t)f(x,y,\cdot)\|_{L^2_z}&=\|g(x,y,\cdot)\|_{L^2_z} \\
&=\|\int\int\frac{e^{i((x-x_1)^2+(x-x_2)^2))/4|t|}}{4 \pi t}
\widehat{f}^{(x_3)}(x_1,x_2,\cdot)dx_1dx_2\|_{L^2_z}\\
&\leq \int\int \|\frac{e^{i((x-x_1)^2+(x-x_2)^2))/4|t|}}{4 \pi t}
\widehat{f}^{(x_3)}(x_1,x_2,\cdot)\|_{L^2_z}dx_1dx_2\\
&\leq \frac{1}{4 \pi |t|}\int\int \|\widehat{f}^{(x_3)}(x_1,x_2,\cdot)\|_{L^2_z}dx_1dx_2\\
&=\frac{1}{4 \pi |t|}\int\int \|{f}(x_1,x_2,\cdot)\|_{L^2_z}dx_1dx_2=\frac{1}{4 \pi |t|} \|{f}\|_{L^1_{xy}L^2_z}.
\end{split}
\end{equation*}
Therefore, from the last inequality we obtain
\begin{align}\label{q2}
 \|\E(t)f\|_{L^\infty_{xy}L^2_z}\leq \frac{1}{4 \pi |t|} \|
{f}\|_{L^1_{xy}L^2_z}.
\end{align}
Interpolation between inequalities \eqref{q1} and \eqref{q2} yields the result.
\end{proof}

Now we are able to prove Strichartz estimates. We notice that our result do not cover the endpoint $(p,q)=(\infty,2).$
%%%%% TEO STRIC
\begin{proposition}[Strichartz estimates]\label{stric}
 The unitary group $\{\E(t)\}_{-\infty}^{+\infty}$ defined in  \eqref{sol} satisfies
\begin{align}\label{stric1}
 \|\E(t)f\|_{L^q_tL^p_{xy}L^2_z}\leq c\|f\|_{L^2_{xyz}},
\end{align}
\begin{align}\label{stric2}
 \|\int_{\re}\E(t-t')g(\cdot,t')dt'\|_{L^q_tL^p_{xy}L^2_z}\leq c\|g\|_{L^{q'}_tL^{p'}_{xy}L^2_z}
\end{align}
and
\begin{align}\label{stric3}
\|\int_{\re}\E(t)g(\cdot,t)dt\|_{L^2_{xyz}}\leq c\|g\|_{L^{q'}_tL^{p'}_{xy}L^2_z},
\end{align}
where $$\frac{1}{p}+\frac{1}{p'}=\frac{1}{q}+\frac{1}{q'}=1,\  \frac{2}{q}=1-\frac{2}{p}\text{ and } p=\frac{2}{\theta},\ 
\theta \in (0,1].$$
\end{proposition}

\begin{proof} To prove this proposition we use standard by now arguments. First one shows that the
three inequalities are equivalent. The main ingredient is the Stein-Thomas argument. Thus it
is enough to establish for instance the estimate \eqref{stric2}. To obtain \eqref{stric2}  we use Lemma \ref{lj} and
the Hardy-Littlewood-Sobolev theorem.
\end{proof}

Next we recall some  estimates proved in \cite{LiPoS} regarding the solutions of the linear problem
\begin{equation}\label{eqondalinear}
\begin{cases}
\partial_t^2 n + \Delta_\perp n  = 0, \;\;\;\;(x,y,z) \in \re^3,\;\; t>0,\\
n(\cdot, 0)=n_0(\cdot)\\
  \partial_t n(\cdot, 0)=n_1(\cdot),
\end{cases}
\end{equation}
where $\Delta_\perp=\partial^2_x + \partial^2_y$. The solution of the problem \eqref{eqondalinear} can be written as
\begin{equation}\label{n}
 n(\cdot, t)=N'(t)n_0+N(t)n_1,
\end{equation}
where $N(t)$ and $N'(t)$ were defined in  \eqref{defN} and \eqref{defN'}.
\begin{lemma}\label{l2}
 For $f \in L^2(\re^3)$ we have
\begin{align}
 \|N(t)f\|_{L^2(\re^3)}\leq|t| \|f\|_{L^2(\re^3)},\label{LE4}
\end{align}
\begin{align}
\|N'(t)f\|_{L^2(\re^3)}\leq \|f\|_{L^2(\re^3)},\label{LE5}
\end{align}
and 
\begin{align}
\|(-\Delta_\perp)^{1/2}N(t)f\|_{L^2(\re^3)}\leq \|f\|_{L^2(\re^3)}\label{LE6}.
\end{align}
\end{lemma}

\begin{remark}\label{reml2}
From this lemma one can easily deduce that 
\begin{equation}\label{waveinH2}
\sum_{|\alpha|\leq2}\|N(t)\partial^{\alpha}f\|_{L^2(\R^3)} \leq c\|f\|_{H^1(\re^3)}+
c|t|\|\partial_zf\|_{H^1(\re^3)}.
\end{equation}
\end{remark}

\begin{lemma}\label{l1}
\begin{align}
 \|N'(t)n_0\|_{L^2_xL^{\infty}_{yzT}}\leq  \|n_0\|_{H^2(\re^3)},\label{LE7}
\end{align}

\begin{align}
\|(-\Delta_\perp)^{1/2}N(t)n_1\|_{L^2_{x}L^{\infty}_{yzT}}\leq T\|n_1\|_{H^2(\re^3)},\label{LE8}
\end{align}
and
\begin{align}
\|N(t)n_1\|_{L^2_{x}L^{\infty}_{yzT}}\leq T(\|n_1\|_{H^1(\re^3)}+\|\partial_z n_1\|_{H^1(\re^3)})\label{LE9}.
\end{align}
These estimates hold exchanging x and y. 
\end{lemma}

In our argument we shall use some of the calculus inequalities involving fractional derivatives proved in \cite{KPV}.
More precisely, we recall the following estimate which is a particular case of those established in
(\cite{KPV}, Theorem A.8).
\begin{lemma}\label{fracder} Let $\rho\in(0,1)$, $\rho_1, \rho_2\in [0,\rho]$ with $\rho=\rho_1+\rho_2$. Furthermore, let $p_1, p_2, q_1, q_2\in [2,\infty)$
 such that
 $$
 \frac12=\frac{1}{p_1}+\frac{1}{p_2}=\frac{1}{q_1}+\frac{1}{q_2}.
 $$
 Then
\begin{equation}
  \|D^{\rho}_{x_j}(fg)-fD^{\rho}_{x_j}g-D^{\rho}_{x_j}f g\|_{L^2_{x_j}(\R;L^2(Q))}\leq c\,\|D^{\rho_1}_{x_j}f\|_{L^{p_1}_{x_j}(\R;L^{q_1}(Q))}\|D^{\rho_2}_{x_j}g\|_{L^{p_2}_{x_j}(\R;L^{q_2}(Q))},
\end{equation}
where $Q=\R^{n-1}\times[0,T]$.
\end{lemma}

%%%%%%%%%%%%%%%%%%%%%%%%%%%%%%%%%%%%%%%%%%%%%%%%%%%%%%%%%%%%
%%%%%%%%%%%%%%%%%%%%%%%%%%%%%%%%%%%%%%%%%%%%%%%%%%%%%%%%%%%%

\section{Nonlinear Estimates}\label{nonlinearestimates}
In this section we will establish estimates for the nonlinear terms involving in our analysis

We begin by rewriting the integral equivalent form  of the IVP \eqref{eqzd1} as
\begin{equation}\label{eqint2}
 E(t)=\E(t)E_0 +\int_{0}^{t}\E(t-t')(EF)(t')dt'+\int_{0}^{t}\E(t-t')(EL)(t')dt',
\end{equation}
where
\begin{equation}\label{defF}
 F(t)=N'(t)n_0+N(t)n_1,
\end{equation}
and
\begin{equation}\label{defH}
 L(t)=\int_{0}^{t}N(t-t')\Delta_\perp(|E|^2)(t')dt'.
\end{equation}

In the next lemma we treat the nonlinearity $L$ in the Sobolev norm $\|\cdot\|_{H^2}$. 

\begin{lemma}\label{lemmaH1} Let $\alpha, \beta_1, \beta_2$ be multi-indices, then
\begin{equation}\label{NLE1}
\begin{split}
 \sum_{|\alpha|\leq 2}\|\partial^{\alpha}L\|_{L^2_{xyz}}
\leq & c\, T^{1/2}\|{E}\|_{L^{2}_xL^{\infty}_{yzT}} \underset{|\alpha|=2}
{ \sum}\|\partial_x\partial^{\alpha}E\|_{L^{\infty}_xL^2_{yzT}}
+c\, T^{1/2}\|{E}\|_{L^{2}_yL^{\infty}_{xzT}} \underset{|\alpha|=2}
{ \sum}\|\partial_y\partial^{\alpha}E\|_{L^{\infty}_yL^2_{xzT}}\\
&+cT^{1/2}\underset{|\beta_2|=1}{\underset{|\beta_1|+|\beta_2| \le 2}{\sum}}
\Big(\|\partial_x\partial^{\beta_1}E\|_{L^4_{xyT}L^{2}_{z}}+\|\partial_y\partial^{\beta_1}E\|_{L^4_{xyT}L^{2}_{z}}\Big)
\|J_z^{\frac{1}{2}+}\partial^{\beta_2}\bar{E}\|_{L^4_{xyT}L^{2}_{z}}\\
&+cT\|E\|^2_{L^{\infty}_TH^2(\R^3)}.
\end{split}
\end{equation}

\end{lemma}
\begin{proof}
Using the definition of $L$ in \eqref{defH} and the inequality \eqref{LE6} we have that
\begin{equation}\label{c0}
 \begin{split}
 \sum_{|\alpha|\leq 2}\|\partial^{\alpha}L\|_{L^2_{xyz}}
&\leq c\sum_{|\alpha|\leq2} \int_0^T\|(-\Delta_\perp)^{1/2}N(t'-s)(-\Delta_\perp)^{1/2}\partial^{\alpha}
(|E|^2)(s)\|_{L^2_{xyzT}}ds\\
&\leq c\sum_{|\alpha|\leq 2}\int_0^T\|(-\Delta_\perp)^{1/2}\partial^{\alpha}(|E|^2)(s)\|_{L^2_{xyz}}ds\\
\leq cT\|E\|^2_{L^{\infty}_TH^2(\R^3)}&+ cT^{1/2} \sum_{|\alpha|=2}\Big(\|\partial_x\partial^{\alpha}(E\bar{E})\|_{L^2_{xyzT}}+
\|\partial_y\partial^{\alpha}(E\bar{E})(s)\|_{L^2_{xyzT}}\Big).
\end{split}
\end{equation}

Next it will be enough to consider one of the terms inside the sum on the right hand side of \eqref{c0}.
By Leibniz' rule  we have
\begin{equation}\label{sum2}
\begin{split}
\|\partial_x\partial^\alpha(E\bar{E})\|_{L^2_{xyzT}} 
&\leq  c\underset{|\beta_1| +|\beta_2|= 2}{\sum}\|\partial_x(\partial^{\beta_1}E\partial^{\beta_2}\bar{E})\|_{L^2_{xyzT}}\\
&\leq c \underset{ |\beta_1| +|\beta_2|=2}{\sum} (\|\partial_x\partial^{\beta_1}E\partial^{\beta_2}\bar{E}\|_{L^2_{xyzT}}
+\|\partial^{\beta_1}E\partial_x\partial^{\beta_2}\bar{E}\|_{L^2_{xyzT}}\Big).
\end{split}
\end{equation}

Now we will consider just one of the terms on the right hand of the last inequality, the other one can be similarly treated.
To simplify the exposition we choose the terms $\|\partial_x^3E\bar{E}\|_{L^2_{xyzT}}$ and
$\|\partial_x^2E\partial_x\bar{E}\|_{L^2_{xyzT}}$ to show the next estimates since they have the same structure of the reminder terms
in the sum in \eqref{sum2}.

The H\"older inequality implies
\begin{equation}\label{term1}
\|\partial_x^3E\bar{E}\|_{L^2_{xyzT}}\le \|\partial_x^3E\|_{L^{\infty}_xL^2_{yzT}} \|E\|_{L^2_xL^{\infty}_{xyzT}}.
\end{equation}

On the other hand, using the H\"older inequality and the Sobolev lemma in the $z$-direction we obtain
\begin{equation}\label{term2}
\begin{split}
\|\partial_x^2E\partial_x\bar{E}\|_{L^2_{xyzT}}&\le \|\|\partial_x^2 E\|_{L^2_z}\|\partial_xE\|_{L^{\infty}_z}\|_{L^2_{xyT}}\\
&\le c\, \|\|\partial_x^2 E\|_{L^2_z}\|J_z^{\frac12+}\partial_xE\|_{L^2_z}\|_{L^2_{xyT}}\\
&\le c\|\partial_x^2 E\|_{L^4_{xyT}L^2_z}\|J_z^{\frac12+}\partial_xE\|_{L^4_{xyT}L^2_z}.
\end{split}
\end{equation}

Using the information in inequalities \eqref{term1} and \eqref{term2} in \eqref{sum2}  and then in \eqref{c0} the estimate
\eqref{NLE1} follows.

\end{proof}

\begin{lemma}\label{lemmaH3} Let $\alpha, \beta_1, \beta_2$ be multiindexes, then
\begin{equation}\label{NLE2}
\begin{split}
 \|L&\|_{L^2_xL^{\infty}_{yzT}}+\|L\|_{L^2_yL^{\infty}_{xzT}} \leq cT^2\|E\|^2_{L^{\infty}_TH^2(\R^3)} \\
& + c\, T^{3/2} \|{E}\|_{L^{2}_xL^{\infty}_{yzT}}
\underset{|\alpha|=2}{ \sum}\|\partial_x\partial^{\alpha}E\|_{L^{\infty}_xL^2_{yzT}}+
c\, T^{3/2} \|{E}\|_{L^{2}_yL^{\infty}_{xzT}}
\underset{|\alpha|=2}{ \sum}\|\partial_y\partial^{\alpha}E\|_{L^{\infty}_yL^2_{xzT}}
\\
&+cT^{3/2}\underset{|\beta_2|=1}{\underset{|\beta_1|+|\beta_2| \le 2}{\sum}}
\Big(\|\partial_x\partial^{\beta_1}E\|_{L^4_{xyT}L^{2}_{z}}+\|\partial_y\partial^{\beta_1}E\|_{L^4_{xyT}L^{2}_{z}}\Big)
\|J_z^{\frac{1}{2}+}\partial^{\beta_2}\bar{E}\|_{L^4_{xyT}L^{2}_{z}}.
\end{split}
\end{equation}
\end{lemma}

\begin{proof}
Using Lemma \ref{l2} we have that
 \begin{equation*}
 \begin{split}
 \|L\|_{L^2_{x}L^{\infty}_{yzT}}&\leq \int_0^T\|(-\Delta_\perp)^{1/2}N(t-t')(-\Delta_\perp)^{1/2}(|E|^2)(t')
 \|_{L^2_{x}L^{\infty}_{yzT}}dt'\\
 &\leq T\int_0^T \|(-\Delta_\perp)^{1/2}(|E|^2)(t')\|_{H^2}dt'.
 \end{split}
 \end{equation*}
 Thus applying the argument used in Lemma \ref{lemmaH1} the result follows.
\end{proof}

                      %%%%%%%%%%%%%%%%%%%%%%%%%%%%%%%%%%%%%%%%%%%%%

\begin{lemma}\label{lemmaEF}
 \begin{equation}\label{NLE}
 \begin{split}
 \sum_{|\alpha|\leq 2} \|{\partial}^{\alpha}(EF)\|_{L^2_{xyzT}}\leq c T^{1/2}\,\|E\|_{L^{\infty}_TH^2} 
\big(\|n_0\|_{H^2}+\|n_1\|_{H^1}+T\|\partial_zn_1\|_{H^1}\big),
 \end{split}
 \end{equation}
and
 \begin{equation}\label{NLE0}
 \begin{split}
&\underset{|\alpha|\le 2}{\sum} \|{\partial}^{\alpha}(EL)\|_{L^2_{xyzT}}\leq 
cT^{3/2}\|E\|^3_{L^{\infty}_TH^2(\R^3)}\\
&+cT^{1/2}\,\underset{|\beta_2|=1}{\underset{|\beta_1|+|\beta_2| \le 2}{\sum}}
\Big(\|\partial_x\partial^{\beta_1}E\|_{L^4_{xyT}L^{2}_{z}}+\|\partial_y\partial^{\beta_1}
E\|_{L^4_{xyT}L^{2}_{z}}\Big)
\|J_z^{\frac{1}{2}+}\partial^{\beta_2}\bar{E}\|_{L^4_{xyT}L^{2}_{z}}\\
 & +cT\|E\|_{L^{\infty}_TH^2(\R^3)}\big(\|{E}\|_{L^{2}_xL^{\infty}_{yzT}} \underset{|\alpha|=2}{ \sum}
 \|\partial_x\partial^{\alpha}E\|_{L^{\infty}_xL^2_{yzT}}
+\|{E}\|_{L^{2}_yL^{\infty}_{xzT}} \underset{|\alpha|=2}{ \sum} \|\partial_y\partial^{\alpha} E\|_{L^{\infty}_yL^2_{xzT}}\big)
.
\end{split}
 \end{equation}
\end{lemma}
\begin{proof}
 To obtain the estimate \eqref{NLE} we first use properties of Sobolev spaces  to obtain
\begin{equation}
\underset{|\alpha|\le 2}{\sum}\|{\partial}^{\alpha}(EF)\|_{L^2_{xyzT}}\leq  
 c T^{1/2}\|E\|_{L^{\infty}_TH^2(\R^3)}\|F\|_{L^{\infty}_TH^2(\R^3)}.
\end{equation}
Then Lemma \ref{l2} and  \eqref{waveinH2}  yield the result.

 Analogously to obtain the estimate \eqref{NLE0} we first use properties of Sobolev spaces  to obtain
\begin{equation}
\underset{|\alpha|\le 2}{\sum}\|{\partial}^{\alpha}(EL)\|_{L^2_{xyzT}}\leq 
 c T^{1/2}\|E\|_{L^{\infty}_TH^2(\R^3)}\|L\|_{L^{\infty}_TH^2(\R^3)}.
\end{equation}
Then Lemma \ref{lemmaH1}  yields the result.
\end{proof}
                                    
 %%%%%%%%%%%%%%%%%%%%%%%%%%%%%%%%%%%%%%%%%%%%%%%%
\begin{lemma}\label{lemmaEF1}
\begin{equation}
\begin{split}
 \sum_{|\alpha|=2}\|{\partial}_x\int_{0}^{t} &\E(t-t'){\partial}^{\alpha}(EF)(t')dt'\|_{L^{\infty}_xL^{2}_{yzT}}\\
&\leq c\,T^{1/2}\|E\|_{L^\infty_TH^2}\big(\|n_0\|_{H^2}+T\|n_1\|_{H^1}+T\|\partial_zn_1\|_{H^1}\big)\\
&\;\;\;+ c\,T^{1/2} \|E\|_{L^2_xL^{\infty}_{yzT}}\big(\|n_0\|_{H^2}+\|n_1\|_{H^1}+T^{1/2}\|\partial_zn_1\|_{H^1}\big)\\
&\;\;\;+ \underset{|\beta_1|=1}{\sum} \big(T^{3/4}\|J^{1/4+}_zD^{1/2}_x\partial^{\beta_1}E\|_{L^4_{xyT}L^2_{z}}+T^{5/8}\|J_z^{3/8+}\partial^{\beta_1}E\|_{L^{8/3}_TL^8_{xy}L^2_{z}}\big)\\
&\;\;\;\times \big(\|n_0\|_{H^2}+T^{1/2}\|n_1\|_{H^1}+T\|\partial_zn_1\|_{H^1}\big).
\end{split}
\end{equation}
These estimates holds exchanging x and y.
\end{lemma}

\begin{proof} 
Let $\beta_i\in (\Z^{+})^3$, $i=1,2$, be multi-indices.  The Leibniz rule and Proposition \ref{p1} yield
 \begin{equation}\label{l4A}
 \begin{split}
\|{\partial}_x\int_{0}^{t} \E(t-t'){\partial}^{\alpha}(EF)&(t')dt'\|_{L^{\infty}_xL^{2}_{yzT}}\\
&\lesssim
\underset{|\alpha|=2}{\sum} \|{\partial}_x\int_{0}^{t}  \E(t-t')(\partial^{\alpha}E\,F+ E\,\partial^{\alpha}F)(t')dt'\|_{L^{\infty}_xL^{2}_{yzT}}\\
&\;\;+\underset{|\beta_1|=|\beta_2|=1}{\sum} \|{\partial}_x\int_{0}^{t}  \E(t-t')(\partial^{\beta_1}E\partial^{\beta_2}F)(t')dt'\|_{L^{\infty}_xL^{2}_{yzT}}\\
&\lesssim \underset{|\alpha|=2}{\sum}\big( \|\partial^{\alpha}E\,F\|_{L^1_xL^{2}_{yzT}}+ \|E\,\partial^{\alpha}F\|_{L^1_xL^{2}_{yzT}}\big)\\
&\;\;+\underset{|\beta_1|=|\beta_2|=1}{\sum} \int_{0}^{T}  \|D_x^{1/2}(\partial^{\beta_1}E\partial^{\beta_2}F)\|_{L^2_{xyz}}.
\end{split}
\end{equation}

Using the Holder inequality  Lemma \ref{l1} and Remark \ref{reml2} we deduce that
\begin{equation}\label{l4B}
\begin{split}
\underset{|\alpha|=2}{\sum}\big( \|\partial^{\alpha}E\,F\|_{L^1_xL^{2}_{yzT}}&+ \|E\,\partial^{\alpha}F\|_{L^1_xL^{2}_{yzT}}\big)\\
&\le  c\,T^{1/2}\|E\|_{L^\infty_TH^2}\big(\|n_0\|_{H^2}+T\|n_1\|_{H^1}+T\|\partial_zn_1\|_{H^1}\big)\\
&\;\;\;+ c\, \|E\|_{L^2_xL^{\infty}_{yzT}}\big(\|n_0\|_{H^2}+T^{1/2}\|n_1\|_{H^1(\re^3)}+T\|\partial_zn_1\|_{H^1(\re^3)}\big).
\end{split}
\end{equation}

On the other hand,  the use of the fractional Leibniz rule \eqref{fracder},  the Holder inequality and the Sobolev  embedding in several stages yield the
next chain of inequalities
\begin{equation}\label{l4C}
\begin{split}
&\underset{|\beta_1|=|\beta_2|=1}{\sum} \int_{0}^{T}  \|D_x^{1/2}(\partial^{\beta_1}E\partial^{\beta_2}F)\|_{L^2_{xyz}}\\
&\leq c\,\underset{|\beta_1|=|\beta_2|=1}{\sum} \int_{0}^{T}\big(\|D^{1/2}_x\partial^{\beta_1}E(t')\|_{L^4_{xyz}}\|\partial^{\beta_2}F(t')\|_{L^{4}_{xyz}}
+\|\partial^{\beta_1}E(t')D^{1/2}_x\partial^{\beta_2}F(t')\|_{L^2_{xyz}}\big)dt'\\
&\leq c\,\underset{|\beta_1|=1}{\sum} \big(T^{3/4}\|J^{1/4+}_zD^{1/2}_x\partial^{\beta_1}E\|_{L^4_{xyT}L^2_{z}}\|F\|_{L^\infty_{T}H^{2}}
+T^{5/8}\|J_z^{3/8+}\partial^{\beta_1}E\|_{L^{8/3}_TL^8_{xy}L^2_{z}}\|F\|_{L^\infty_{T}H^{2}}
\big).\\
&\leq c\,\underset{|\beta_1|=1}{\sum} \big(T^{3/4}\|J^{1/4+}_zD^{1/2}_x\partial^{\beta_1}E\|_{L^4_{xyT}L^2_{z}}+T^{5/8}\|J_z^{3/8+}\partial^{\beta_1}E\|_{L^{8/3}_TL^8_{xy}L^2_{z}}\big)\\
&\;\;\;\times \big(\|n_0\|_{H^2}+T^{1/2}\|n_1\|_{H^1(\re^3)}+T\|\partial_zn_1\|_{H^1(\re^3)}\big).
\end{split}
\end{equation}

Thus combining \eqref{l4B}, \eqref{l4C} and \eqref{l4A} the result follows.
\end{proof}

%%%%%%%%%%%%%%%%%%%%%%%%%%%%%%%%%%%%%%%%%%%%%%%%%%
\begin{lemma}\label{lemmaH2}
\begin{equation*}
\begin{split}
 \sum_{|\alpha|=2}\|{\partial}_x\int_{0}^{t}
& \E(t-t'){\partial}^{\alpha}(EL)(t')dt'\|_{L^{\infty}_xL^{2}_{yzT}}\leq c\,T^{1/2}\|E\|_{L^\infty_TH^2}\|L\|_{L^2_xL^{\infty}_{yzT}} + cT^{1/2}\,\|E\|_{L^2_xL^{\infty}_{yzT}}\|L\|_{L^{\infty}_TH^2 }.\\
&\;\;+ c\,\underset{|\beta_1|=1}{\sum} \big(T^{3/4}\|J^{1/4+}_zD^{1/2}_x\partial^{\beta_1}E\|_{L^4_{xyT}L^2_{z}}
+T^{5/8}\|J_z^{3/8+}\partial^{\beta_1}E\|_{L^{8/3}_TL^8_{xy}L^2_{z}})\|L\|_{L^\infty_{T}H^{2}}.
\end{split}
\end{equation*}
The estimate holds exchanging x and y.
\end{lemma}

\begin{proof} We follow the argument in the previous lemma. More precisely, 
Let $\beta_i\in (\Z^{+})^3$, $i=1,2$, be multi-indices.  The Leibniz rule and Proposition \ref{p1} yield
 \begin{equation}\label{l5A}
 \begin{split}
\|{\partial}_x\int_{0}^{t} \E(t-t'){\partial}^{\alpha}(EL)&(t')dt'\|_{L^{\infty}_xL^{2}_{yzT}}\\
&\lesssim
\underset{|\alpha|=2}{\sum} \|{\partial}_x\int_{0}^{t}  \E(t-t')(\partial^{\alpha}E\,L+ E\,\partial^{\alpha}L)(t')dt'\|_{L^{\infty}_xL^{2}_{yzT}}\\
&\;\;+\underset{|\beta_1|=|\beta_2|=1}{\sum} \|{\partial}_x\int_{0}^{t}  \E(t-t')(\partial^{\beta_1}E\partial^{\beta_2}L)(t')dt'\|_{L^{\infty}_xL^{2}_{yzT}}\\
&\lesssim \underset{|\alpha|=2}{\sum}\big( \|\partial^{\alpha}E\,L\|_{L^1_xL^{2}_{yzT}}+ \|E\,\partial^{\alpha}L\|_{L^1_xL^{2}_{yzT}}\big)\\
&\;\;+\underset{|\beta_1|=|\beta_2|=1}{\sum} \int_{0}^{T}  \|D_x^{1/2}(\partial^{\beta_1}E\partial^{\beta_2}L)\|_{L^2_{xyz}}
\end{split}
\end{equation}

Using the Holder inequality we deduce that
\begin{equation}\label{l5B}
\begin{split}
\underset{|\alpha|=2}{\sum}\big( \|\partial^{\alpha}E\,L\|_{L^1_xL^{2}_{yzT}}&+ \|E\,\partial^{\alpha}L\|_{L^1_xL^{2}_{yzT}}\big)\\
&\le  c\,T^{1/2}\|E\|_{L^\infty_TH^2}\|L\|_{L^2_xL^{\infty}_{yzT}} + cT^{1/2}\,\|E\|_{L^2_xL^{\infty}_{yzT}}\|L\|_{L^{\infty}_TH^2 }.
\end{split}
\end{equation}

On the other hand,  the use of the fractional Leibniz rule \eqref{fracder},  the Holder inequality and the Sobolev  embedding  yield
\begin{equation}\label{l5C}
\begin{split}
&\underset{|\beta_1|=|\beta_2|=1}{\sum} \int_{0}^{T}  \|D_x^{1/2}(\partial^{\beta_1}E\partial^{\beta_2}L)\|_{L^2_{xyz}}\\
&\leq c\,\underset{|\beta_1|=|\beta_2|=1}{\sum} \int_{0}^{T}\big(\|D^{1/2}_x\partial^{\beta_1}E(t')\|_{L^4_{xyz}}\|\partial^{\beta_2}L(t')\|_{L^{4}_{xyz}}
+\|\partial^{\beta_1}E(t')D^{1/2}_x\partial^{\beta_2}L(t')\|_{L^2_{xyz}}\big)dt'\\
&\leq c\,\underset{|\beta_1|=1}{\sum} \big(T^{3/4}\|J^{1/4+}_zD^{1/2}_x\partial^{\beta_1}E\|_{L^4_{xyT}L^2_{z}}
+T^{5/8}\|J_z^{3/8+}\partial^{\beta_1}E\|_{L^{8/3}_TL^8_{xy}L^2_{z}})\|L\|_{L^\infty_{T}H^{2}}.
\end{split}
\end{equation}

\end{proof}

%%%%%%%%%%%%%%%%%%%%%%%%%%%%%%%%%%%%%%%%%%%%%%%%%%%%
\section{Proof of Theorem \ref{T2}}\label{proofoftheoremT2}
%%%%%%%%%%%%%%%%%%%%%%%%%%%%%%%%%%%%%%%%%%%%%%%%%%%%

As we mention in the introduction we will use the contraction mapping principle.

We first define the metric space
\begin{equation*}\label{XaT}
X_{a, T}=\{E \in C([0, T]:\widetilde{H}^2(\re^3)):\vvvert E \vvvert \leq a \},
\end{equation*}
where
\begin{equation}
\begin{split}
\vvvert E \vvvert:=&\|E\|_{L^{\infty}_T H^2(\re^3)}+ \sum_{|\alpha|= 2}\big( \|D^{1/2}_x{\partial}^{\alpha}E \|_{L^{\infty}_TL^2_{xyz}}+ 
\|D^{1/2}_y{\partial}^{\alpha}E\|_{L^{\infty}_TL^2_{xyz}}\big) \\
&+  \underset{|\alpha|=1}{\sum} \big(\|J^{1/4+}_zD^{1/2}_x\partial^{\alpha}E\|_{L^4_{xyT}L^2_{z}}
+\|J_z^{3/8+}\partial^{\alpha}E\|_{L^{8/3}_TL^8_{xy}L^2_{z}}+\|J_z^{1/2+}\partial^{\alpha}E \|_{L^{4}_{xyT}L^2_{z}}\big)\\
&+ \|E \|_{L^{2}_xL^{\infty}_{yzT}} + \|E \|_{L^{2}_yL^{\infty}_{xzT}}\\
&+ \sum_{|\alpha|\leq1}\big(\|\partial_{x}\partial^\alpha E \|_{L^{4}_{xyT}L^2_{z}} + 
\|\partial_{y}\partial^\alpha E \|_{L^{4}_{xyT}L^2_{z}}\big)\nonumber\\
& + \sum_{|\alpha|= 2}\big( \|\partial_x{\partial}^{\alpha}E \|_{L^{\infty}_xL^2_{yzT}}+ 
\|\partial_y{\partial}^{\alpha}E\|_{L^{\infty}_yL^2_{xzT}}\big).
\end{split}
\end{equation}
and  the integral operator on $X_{a,T}$,
\begin{equation}\label{defpsi}
 \Psi(E)(t)=\E(t)E_0+\int_{0}^{t}\E(t-t')(EF)(t')dt'+\int_{0}^{t}\E(t-t')(EL)(t')dt',
\end{equation}
where
$F$ and $L$ were defined in \eqref{defF} and \eqref{defH}, respectively.

We will show that for appropriate $a$ and $T$ the operator $\Psi(\cdot)$ defines a contraction on $X_{a, T}$.

We begin by estimating the $H^2(\re^3)$-norm of $\Psi(E)$. Let $E \in X_{a, T}$. By Fubini's Theorem, Minkowski's inequality and group
 properties we have 
\begin{align}\label{z1}
 \|\Psi(E)(t)\|_{H^2} \leq \|E_0\|_{H^2} +\|E\|_{L^{\infty}_TH^2}\int_0^T\|F(t')\|_{H^2}dt'
 +\|E\|_{L^{\infty}_TH^2}\int_0^T\|L(t')\|_{H^2}dt'.
\end{align}
From  Lemma \ref{l2}, Lemma \ref{lemmaH1} and inequality \eqref{waveinH2} we have
\begin{equation}\label{ineq}
\begin{split}
\|\Psi(E)(t)\|_{H^2} \leq &\|E_0 \|_{H^2}+c T\|E\|_{L^{\infty}_TH^2}
\big(\|n_0\|_{H^2}+\|n_1\|_{H^1(\re^3)}+T\|\partial_zn_1\|_{H^1(\re^3)}\big) \\
&+cT^{3/2}\|E\|^3_{L^{\infty}_TH^2(\R^3)}+c T^{3/2}\|E\|_{L^{\infty}_TH^2} \|{E}\|_{L^{2}_xL^{\infty}_{yzT}} \underset{|\alpha|=2}
{ \sum}\|\partial_x\partial^{\alpha}E\|_{L^{\infty}_xL^2_{yzT}}\\
&+c\, T^{3/2}\|E\|_{L^{\infty}_TH^2}\|{E}\|_{L^{2}_yL^{\infty}_{xzT}} \underset{|\alpha|=2}
{ \sum}\|\partial_y\partial^{\alpha}E\|_{L^{\infty}_yL^2_{xzT}}\\
+cT^{3/2}\|E\|_{L^{\infty}_TH^2}&\underset{|\beta_2|=1}{\underset{|\beta_1|+|\beta_2| \le 2}{\sum}}
\Big(\|\partial_x\partial^{\beta_1}E\|_{L^4_{xyT}L^{2}_{z}}+\|\partial_y\partial^{\beta_1}E\|_{L^4_{xyT}L^{2}_{z}}\Big)
\|J_z^{\frac{1}{2}+}\partial^{\beta_2}\bar{E}\|_{L^4_{xyT}L^{2}_{z}}.
\end{split}
\end{equation}

Therefore
\begin{equation}\label{norm1}
\begin{split}
\|\Psi(E)(t)\|_{L^{\infty}_TH^2} \leq &\|E_0 \|_{H^2}+c T \vvvert E  \vvvert
\big(\|n_0\|_{H^2}+\|n_1\|_{H^1(\re^3)}+T\|\partial_zn_1\|_{H^1(\re^3)}\big) \\
&+cT^{3/2} \vvvert E  \vvvert^3.
\end{split}
\end{equation}

Next, we estimate  the norms 
$$ \|\cdot\|_{L^{2}_xL^{\infty}_{yzT}},\ 
\sum_{|\alpha|\leq 2}\|\partial_x\partial^{\alpha}\cdot\|_{L^{\infty}_xL^{2}_{yzT}},\ 
\|\cdot\|_{L^{2}_yL^{\infty}_{xzT}},\ 
\sum_{|\alpha|\leq 2}\|\partial_y\partial^{\alpha}\cdot\|_{L^{\infty}_yL^{2}_{xzT}}.$$

By symmetry is enough to estimate the first two norms. Thus,  using the definition of $\Psi$ in \eqref{defpsi}, Proposition \ref{propositionA}  and the inequalities \eqref{z1} and \eqref{ineq}
 it follows that
\begin{equation}\label{norm2}
\begin{split}
\|\Psi (E) \|_{L^{2}_xL^{\infty}_{yzT}} \leq &\|E_0 \|_{H^2}+c T \vvvert E  \vvvert
\big(\|n_0\|_{H^2}+\|n_1\|_{H^1(\re^3)}+T\|\partial_zn_1\|_{H^1(\re^3)}\big) \\
&+cT^{3/2} \vvvert E  \vvvert^3.
\end{split}
\end{equation}

Next we use Proposition  \ref{stric} and then the inequalities \eqref{z1} and \eqref{ineq} to obtain
\begin{equation}\label{norm3}
\begin{split}
&\underset{|\alpha|=1}{\sum} \|J^{1/4+}_zD^{1/2}_x\partial^{\alpha}\Psi(E)\|_{L^4_{xyT}L^2_{z}}
+\|J_z^{3/8+}\partial^{\alpha}\Psi(E)\|_{L^{8/3}_TL^8_{xy}L^2_{z}}+\|J_z^{1/2+}\partial^{\alpha}\Psi(E) \|_{L^{4}_{xyT}L^2_{z}}\\
&+ \sum_{|\alpha|\leq1}(\|\partial_{x}\partial^\alpha \Psi(E) \|_{L^{4}_{xyT}L^2_{z}} + 
\|\partial_{y}\partial^\alpha\Psi(E) \|_{L^{4}_{xyT}L^2_{z}})\\
&\le c\big( \|E_0\|_{H^2} +\|E\|_{L^{\infty}_TH^2}\int_0^T\|F(t')\|_{H^2}dt' +\|E\|_{L^{\infty}_TH^2}\int_0^T\|L(t')\|_{H^2}dt'\big).\\
&\le \|E_0 \|_{H^2}+c T \vvvert E  \vvvert
\big(\|n_0\|_{H^2}+\|n_1\|_{H^1(\re^3)}+T\|\partial_zn_1\|_{H^1(\re^3)}\big)+cT^{3/2} \vvvert E  \vvvert^3.
\end{split}
\end{equation}

Now using the definition of $\Psi$ in \eqref{defpsi},  Proposition \ref{p1}
and Lemmas \ref{lemmaEF1} and \ref{lemmaH2} we obtain

\begin{equation}\label{norm4}
\begin{split}
\sum_{|\alpha|= 2}\|{\partial}_x{\partial}^{\alpha}\Psi (E) \|_{L^{\infty}_xL^{2}_{yzT}}
 &\leq  c\,\sum_{|\alpha|= 2}\|D^{1/2}_x{\partial}^{\alpha}E_0\|_{L^{2}} + c (T)\,T\, \vvvert E  \vvvert^3\\
&\;\; +c(T)\,T^{1/2}\,\vvvert E  \vvvert\,\big(\|n_0\|_{H^2}+\|n_1\|_{H^1}+T\|\partial_zn_1\|
_{H^1}\big)\\
\end{split}
\end{equation}

It remains to estimate the norms $\sum\limits_{|\alpha|= 2}\|D^{1/2}_x\partial^{\alpha}\cdot
\|_{L^{\infty}_TL^2_{xyz}}$
and $\sum\limits_{|\alpha|= 2}\|D^{1/2}_y\partial^{\alpha}\cdot\|_{L^{\infty}_TL^2_{xyz}}$.
Once again by symmetry we only estimate  the first one.

Now using the definition of $\Psi$ in \eqref{defpsi},  Proposition \ref{p1}
and Lemmas \ref{lemmaEF1} and \ref{lemmaH2} we get

\begin{equation}\label{norm5}
\begin{split}
\sum\limits_{|\alpha|= 2}\|D^{1/2}_x\partial^{\alpha}\Psi(E)\|_{L^{\infty}_TL^2_{xyz}}\leq &
\sum_{|\alpha|= 2}\|D^{1/2}_x{\partial}^{\alpha}E_0\|_{L^{2}}+ c (T)\,T\, \vvvert E  \vvvert^3\\
&+  c(T)\,T^{1/2}\,\vvvert E  \vvvert\,\big(\|n_0\|_{H^2}+\|n_1\|_{H^1}+T\|\partial_zn_1\|
_{H^1}\big).
\end{split}
\end{equation}

Hence, a suitable choice of $a=a(\|E_0\|_{\tilde{H}^3},T)$ and $T$ ( $T$ sufficiently small depending on
 $\|n_0\|_{\widetilde{H}^3}$,
$\|n_1\|_{\widetilde{H}^2}$ and $\|\partial_zn_1\|_{\widetilde{H}^2}$), we see that $\Psi$ maps $X_{a,T}$
into $X_{a,T}$.

Since the reminder of the proof follows a similar argument we will omit it.

\subsection*{Acknowledgments} The authors would like to thank M. Panthee for given several suggestions to improve the presentation
 of this manuscript. F. L. was partially supported by CNPq and FAPERJ/ Brazil.

\end{document}